\documentclass[psamsfonts, pdftex, twoside]{amsart}
\usepackage{amssymb, url, color, enumerate, graphicx}
\usepackage[all]{xy}
\usepackage[utf8]{inputenc}
\newcommand{\llcurly}{<}

\newtheorem{theorem}{Theorem}[section]
\newtheorem{lemma}[theorem]{Lemma}

\newtheorem*{theorem*}{Theorem}

\theoremstyle{definition}
\newtheorem*{definition}{Definition}

\theoremstyle{remark}
\newtheorem{remark}[theorem]{Remark}
\newtheorem{question}[theorem]{Question}
\newtheorem*{remark*}{Remark}

\numberwithin{equation}{section}

\DeclareMathOperator{\asdim}{asdim}

\DeclareMathOperator{\mesh}{mesh}

\DeclareMathOperator{\diam}{diam}

\DeclareMathOperator{\ulim}{ulim}
\DeclareMathOperator{\udim}{udim}
\DeclareMathOperator{\mult}{mult}
\DeclareMathOperator{\supp}{supp}
\DeclareMathOperator{\Int}{Int}

\newcommand{\U}[1]{{\mathcal{U}_{#1}}}
\newcommand{\V}[1]{{\mathcal{V}_{#1}}}

\newcommand{\N}[1]{{\mathcal{N}_{#1}}}

\newcommand{\df}[1]{\em#1\em}

\begin{document}

\bibliographystyle{abbrv}

\title{A new construction of universal spaces for asymptotic dimension}
\author{G. C. Bell}
\address{Department of Mathematics and Statistics,
University of North Carolina at Greensboro,
116 Petty Building, Greensboro, NC, 27412, USA}
\email{gcbell@uncg.edu}
\author{A.Nag\'orko}
\address{
Faculty of Mathematics, Informatics, and Mechanics,
University of Warsaw,
Banacha 2,
02-097 Warszawa,
Poland
}
\email{amn@mimuw.edu.pl}






\begin{abstract}
For each $n$, we construct a separable metric space $\mathbb{U}_n$ that is universal in the coarse category of separable metric spaces with asymptotic dimension ($\asdim$) at most $n$ and universal in the uniform category of separable metric spaces with uniform dimension ($\udim$) at most $n$. Thus, $\mathbb{U}_n$ serves as a universal space for dimension $n$ in both the large-scale and infinitesimal topology. More precisely, we prove: 
  \[
    \asdim \mathbb{U}_n = \udim \mathbb{U}_n = n
  \]
  and such that for each separable metric space $X$,
  \vspace{0.5mm}
  \begin{enumerate}
  \item[a)] if $\asdim X \leq n$, then $X$ is coarsely equivalent to a
    subset of $\mathbb{U}_n$;
  \item[b)] if $\udim X \leq n$, then $X$ is uniformly homeomorphic to
    a subset of $\mathbb{U}_n$.
  \end{enumerate}
\end{abstract}
\maketitle
\thispagestyle{empty}
\section{Introduction}

For each $n$, we construct a space that is universal in the coarse category of separable metric spaces with asymptotic dimension at most $n$.  Such spaces were previously constructed by Dranishnikov and Zarichnyi in~\cite{dranishnikovzarichnyi2004}. Our aim is to give a more transparent construction that highlights the micro-macro analogy between small and large scales.

Our main result is the following theorem.

\begin{theorem}\label{thm:main}
  For each $n$, there exists a separable metric space $\mathbb{U}_n$
  such that 
  \[
    \asdim \mathbb{U}_n = \udim \mathbb{U}_n = n
  \]
  and such that for each separable metric space $X$ the following
  conditions are satisfied.
  \begin{enumerate}
  \item[a)] If $\asdim X \leq n$, then $X$ is \emph{coarsely
      equivalent} to a subset of $\mathbb{U}_n$.
  \item[b)] If $\udim X \leq n$, then $X$ is \emph{uniformly
      homeomorphic} to a subset of~$\mathbb{U}_n$.
  \end{enumerate}
\end{theorem}

The main instrument that we use is a uniform limit, which for our purposes serves the role of a large-scale analog of an inverse limit. In the first part of the paper we prove that a canonical map into a uniform limit of a suitably chosen anti-\v{C}ech approximation of a space is a coarse equivalence. In the second part we construct a sequence into which every such anti-\v{C}ech approximation of an at most $n$-dimensional space embeds isometrically. At small scales the same arguments apply.

The space $\mathbb{U}_n$ that we construct is self-similar in small and large scales. In particular, the asymptotic cone of $\mathbb{U}_n$ is isometric to $\mathbb{U}_n$ (see Remark~\ref{rem:asymptotic cone}). We leave the following two questions open.

\begin{question}
  Let $\mathbb{U}_n$ be the space constructed in the proof of Theorem~\ref{thm:main}.
  \begin{enumerate}
    \item Is $\mathbb{U}_n$ an $n$-dimensional Polish absolute extensor in dimension $n$?
    \item Is $\mathbb{U}_n$  strongly universal in dimension $n$, i.e., is every map from an $n$-dimensional Polish space into $\mathbb{U}_n$ approximable by embeddings? 
  \end{enumerate}

  Characterization theorems of~\cite{levin2006, nagorkophd} state that any space with the above properties is homeomorphic to the universal $n$-dimensional N\"obeling space.
\end{question}


Let $\U{}$ be an open cover of the metric space $X$. Recall that the mesh of $\U{}$ is $\sup\{\diam U\mid U\in\U{}\}$. The multiplicity (or order) of $\U{}$ is the largest $n$ so that there is a collection of $n$ elements of $\U{}$ with non-empty intersection. A number $\epsilon>0$ is said to be a Lebesgue number for $\U{}$ if every subset of $X$ with diameter less than $\epsilon$ is contained in a member of $\U{}$.


The asymptotic Assouad-Nagata dimension, AN-$\asdim$,  is an asymptotic version of the Assouad-Nagata dimension (see, for example \cite{dranishnikovsmith}). For a metric space $X$, we define AN-$\asdim X\le n$ if there is a $c>0$ and an $r_0>0$ so that for each $r\ge r_0$ there is a cover $\U{}$ of $X$ such that $\mesh(\U{})<cr$, $\mult(\U{})\le n+1$, with Lebesgue number greater that $r$.  Many results concerning the asymptotic dimension can be transferred to corresponding results concerning asymptotic Assouad-Nagata dimension, so a natural question is the following.

\begin{question}
  Is there an analogous construction for asymptotic Nagata-Assouad dimension?
\end{question}



\section{Preliminaries}


The asymptotic dimension of a metric space was introduced by Gromov \cite{gromov1993} in his study of the large-scale geometry of Cayley graphs of finitely generated groups. The definition given there is a large-scale analog of Ostrand's characterization of covering dimension. Asymptotic dimension is an invariant of coarse equivalence (see below) and Roe has shown that it can be defined not only for metric spaces, but also for so-called coarse spaces \cite{roe2003}, (although we will content ourselves here with separable metric spaces). Extrinsic interest in asymptotic dimension was piqued by G.~Yu \cite{yu1998}, who showed that the famous Novikov higher signature conjecture holds for groups with finite asymptotic dimension. This result has subsequently been strengthened, but nevertheless, asymptotic dimension remains an intrinsically interesting invariant of the asymptotic approach to topology.

\begin{definition}
  Let $X$ be a metric space with a fixed metric $d$ and let $n$ be an
  integer.  We say that
  \begin{enumerate}
  \item the \emph{uniform covering dimension $\udim X$} of $X$ is less
    than or equal to $n$, if for each $r > 0$ there exists an open
    cover of $X$ with mesh smaller than $r$, positive Lebesgue number
    and multiplicity at most $n + 1$.
  \item the \emph{asymptotic dimension $\asdim X$} of $X$ is less than
    or equal to $n$, if for each $r < \infty$ there exists an open
    cover of $X$ with Lebesgue number greater than $r$, finite mesh
    and multiplicity at most $n+1$.
  \end{enumerate}
\end{definition}

{%
\begin{figure}[ht]
\begin{center}
\resizebox{0.70\textwidth}{!}{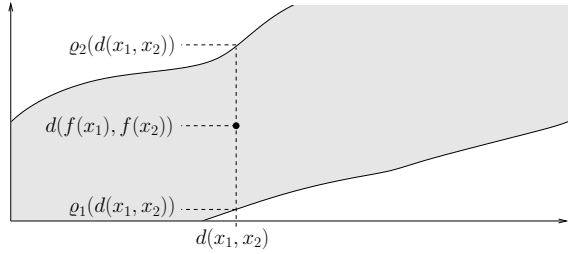}
\end{center}
\caption{{The relation between $\varrho_1$, $f$ and
  $\varrho_2$.}}
\label{coarse_def}
\end{figure}
}

\begin{definition}
  Let $f$ be a map from a metric space $X$ into a metric space $Y$.
  Assume that $\varrho_1, \varrho_2 \colon (0, \infty) \to [0,
  \infty]$ are functions for which the following inequalities hold for each $x_1, x_2 \in X$:
  \[
  \varrho_1(d(x_1,x_2)) \leq d(f(x_1), f(x_2)) \leq
  \varrho_2(d(x_1,x_2)).
  \]
  We say that $f$ is
  \begin{enumerate}
  \item[(I)] \emph{large scale uniform}, if $\varrho_2(r) < \infty$ for each $r$,
  \item[(II)] a \emph{coarse embedding}, if it is large-scale uniform and $\lim_{r
      \to \infty} \varrho_1(r) = \infty$,
  \item[(III)] a \emph{coarse equivalence} if it is a coarse embedding whose
    image is $r$-dense in $Y$ for some $r < \infty$,
  \end{enumerate}

  The above notions correspond to the following notions of
  infinitesimal topology. We say that $f$ is
  \begin{enumerate}
  \item[(i)] \emph{uniformly continuous}, if $\lim_{r \to 0} \varrho_2(r) =
    0$,
  \item[(ii)] a \emph{uniform embedding}, if it is uniformly continuous and
    $\varrho_1(r) > 0$ for each~$r$,
  \item[(iii)] \emph{uniform homeomorphism} if it is a uniform embedding onto
    $Y$.
  \end{enumerate}
\end{definition}

\subsection{Uniform complexes}

We let $\ell_2$ denote the Hilbert space of square summable sequences of real numbers, endowed with the $\|\cdot\|_2$ norm. We identify sequences in $\ell_2$ with maps from the natural numbers into the reals. For an element $x$ of $\ell_2$ we let $x(m)$ denote the $m$th coordinate of $x$. The \emph{support $\supp x$} of an element $x$ of $\ell_2$ is the set of indices of non-vanishing coordinates of $x$, i.e. we let \[ \supp x = x^{-1}(\mathbb{R} \setminus \{ 0 \}) = \left\{ m \in \mathbb{N} \colon x(m) \neq 0 \right\}. \]
We say that a set of points in $\ell_2$ is \emph{contiguous} if their supports have non-empty intersection, i.e. if they have a common non-vanishing coordinate.

For $\kappa > 0$, we let
\[
\Delta_\kappa = \left\{ x \in \ell_2 \colon x(m) \geq 0, \sum_m x(m) = \frac{\kappa}{\sqrt{2}}\right\}
\]
and say that $\Delta_\kappa$ is an \emph{infinite simplex of scale
  $\kappa$ in $\ell_2$}. We say that a simplicial complex is
\emph{uniform of scale~$\kappa$}, if it is endowed with the metric of a
subcomplex of~$\Delta_\kappa$. 

The \emph{barycenter} of a finite collection of points in $\ell_2$ is, by definition, their arithmetic mean. The \emph{barycenter} of a simplex of a uniform simplicial complex is the barycenter of the set of its vertices. A \emph{barycenter} of a uniform simplicial complex is a barycenter of any of its faces (vertices
included), i.e. an element of the complex whose non-vanishing coordinates are all equal. 

\begin{lemma}\label{lem:distances in a uniform complex}
  If $\kappa$ is the scale of a uniform simplicial complex, then the
  lengths of its edges are equal to $\kappa$ and the distance between
  its contiguous barycenters is less than $\kappa/\sqrt{2}$. If the
  complex is $n$-dimensional, then the distance between its disjoint
  simplices is greater than or equal to $\kappa/\sqrt{n+1}$.
\end{lemma}
\begin{proof}
  By definition, a vertex of an infinite simplex of scale $\kappa$
  has exactly one non-vanishing coordinate equal to $\kappa/\sqrt{2}$.
  The distance between any two such points (which is equal to the
  length of the edges) is equal to $\kappa$. 

  By definition, the barycenter of a collection of points is equal
  to their arithmetic mean. Hence the barycenter of an $n$-dimensional
  simplex of $\Delta_\kappa$ has exactly $n+1$ non-vanishing
  coordinates, all equal to $\kappa/(n+1)\sqrt{2}$. Let $\sigma_1$ be
  the barycenter of a $k$-dimensional simplex of $\Delta_\kappa$ and let
  $\sigma_2$ be the barycenter of an $l$-dimensional simplex of
  $\Delta_\kappa$. If $\sigma_1$ and $\sigma_2$ are contiguous, then the
  number $m$ of common non-vanishing coordinates of $\sigma_1$ and
  $\sigma_2$ is greater than $0$. A direct computation shows that
  \[
  \left\| \sigma_1 - \sigma_2 \right\|_2 =
  \frac{\kappa}{\sqrt{2}}\sqrt{\frac{1}{k+1} + \frac{1}{l+1} -
    \frac{2m}{(k+1)(l+1)}}.
  \]
  We can put $m = 1$ and verify that this norm does not exceed
  $\kappa/\sqrt{2}$.

  A vector that connects barycenters of two disjoint simplices in
  $\Delta_\kappa$ is the only vector simultaneously orthogonal to the affine subspaces spanned by them, hence the distance between disjoint $k$- and $l$-dimensional simplices is equal to the
  distance between their barycenters, which by the above formula (we
  have $m = 0$) is equal to $ (\kappa/\sqrt{2})\sqrt{1/(k+1) +
    1/(l+1)}$.  If $k$ and $l$ are not greater than $n$, then the
  minimal value of this distance is equal to $\kappa/\sqrt{n+1}$.
\end{proof}

\subsection{Barycentric maps}

Let $K$ and $L$ be uniform complexes in $\ell_2$. We say that a map from $K$ into $L$ is \df{barycentric} if it is affine, if it maps vertices of~$K$ to barycenters of~$L$ and if it has the following \df{shrinking} property.

\vspace{2mm}
\begin{tabular}{ll}
  {\sc (sh)} & \begin{tabular}{l}
  If vertices of $M \subset K$ are adjacent to a single vertex, then \\
  their images in $L$ have a common non-vanishing coordinate.
                \end{tabular}
\end{tabular}
\vspace{2mm}

Note that every simplicial map that is barycentric must be constant on connected components of the domain.

Let $\U{}$ be a collection of sets. The \emph{nerve} $N(\U{})$ of $\U{}$ is a uniform simplicial complex in $\ell_2$ whose vertices are in a one-to-one correspondence with $\U{}$ and whose vertices span a simplex if and only if the corresponding elements in~$\U{}$ have non-empty intersection. We let $v(U)$ denote the vertex of $N(\U{})$ that corresponds to an element $U$ in $\U{}$.

Let $\U{}$ be a refinement of a cover $\V{}$. We say that a map from $N(\U{})$ into $N(\V{})$
is a \emph{map induced by inclusions}, if it is affine and if for each $U
\in \U{}$ it maps the vertex $v(U)$ of $N(\U{})$ to the
barycenter of a simplex of $N(\V{})$ that is spanned by vertices $v(V)$ 
corresponding to sets $V \in \V{}$ that contain $U$ as a subset.

\begin{lemma}\label{lem:lipschitz on closed stars}
  If $\U{}$ is a
star-refinement of $\V{}$, then the map from $N(\U{})$ into $N(\V{})$ that is induced by inclusions  is barycentric.
  A barycentric map between uniform complexes of the same scale is
  $1/\sqrt{2}$-Lipschitz on closed stars of vertices.
\end{lemma}
\begin{proof}
  Let $U$ be an element of $\U{}$. Since $\U{}$ star-refines $\V{}$, there is an element~$V$ of~$\V{}$ that contains all elements of $\U{}$ that intersect $U$. By definition, a vertex of $N(\U{})$ that corresponds to an element of $\U{}$ that intersects $U$ is mapped by the map induced by inclusions to a barycenter of a simplex that contains $v(V)$. Hence the intersection of the supports of these barycenters contains the support of $v(V)$, so it is non-empty; hence these barycenters are contiguous and the map induced by inclusions is barycentric.

  Since a barycentric map is affine, to prove the second assertion of the lemma it suffices to check that a barycentric map
  shrinks distances between vertices of closed stars by a factor of
  $1/\sqrt{2}$. By definition, a barycentric map maps any two
  vertices of a closed star of a vertex in the domain into contiguous
  barycenters in the codomain. By Lemma~\ref{lem:distances in a
    uniform complex}, the distance between any two vertices in a
  complex of scale $\kappa$ is equal to $\kappa$ and the distance
  between contiguous barycenters is not greater than $\kappa/\sqrt{2}$.
  We are done.
\end{proof}

\section{Uniform limits}

Let $\mathcal{S} = \{ f_k \colon X_k \to X_{k+1} \}$ be a sequence of
functions of metric spaces. The space of threads
\[
\left\{ (x_k) \in \prod_{k = -\infty}^{\infty}  X_k \colon x_{k+1} = f_k(x_k) \right\}
\]
equipped with the sup metric
\[
d(x,y) = \sup_{k} d_{X_k}(x_k, y_k)
\]
(we allow infinite distances) is called the \emph{uniform limit of
  $\mathcal{S}$} and is denoted by $\ulim \mathcal{S}$.

{%
\begin{figure}[ht]
\begin{center}
\resizebox{0.70\textwidth}{!}{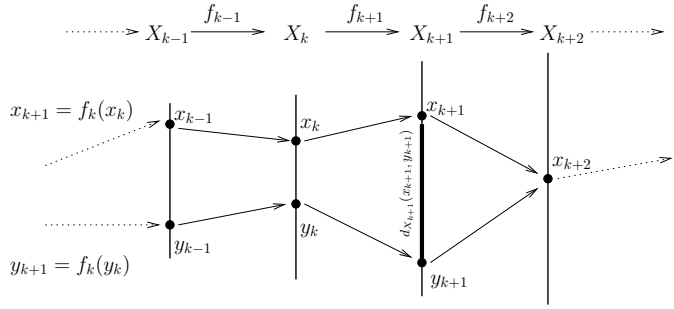}
\end{center}
\caption{{Threads in the uniform limit.}}
\label{threads}
\end{figure}
}

\begin{lemma}\label{lem:complete uniform limits}
  The uniform limit
  \[
    \ulim \xymatrix{ \cdots \ar[r]^{f_{k-2}} & X_{k-1}
    \ar[r]^{f_{k-1}} & X_k \ar[r]^{f_k} & X_{k+1} \ar[r]^{f_{k+1}} &
    \cdots}
  \]
  of a sequence of complete spaces with continuous bonding maps is
  complete.
\end{lemma}
\begin{proof}
  A countable product of complete spaces equipped with the sup
  metric is complete. If the bonding maps are continuous, then the
  space of threads is a closed subspace of the entire product, hence
  it is complete as well.
\end{proof}

\subsection{Uniform sequences}

We say that a sequence
\[
  \mathcal{K} = \xymatrix@C=35pt{
    \cdots \ar[r]^{f_{m-2}} & K_{m-1} \ar[r]^{f_{m-1}} & K_m \ar[r]^{f_m} & K_{m+1} \ar[r]^{f_{m+1}} & \cdots
  }
\]
is \emph{$n$-uniform} if each $K_m$ is an at most $n$-dimensional uniform complex of scale $2^{m/2}$ and each $f_m$ is a barycentric map.

\begin{lemma}\label{lem:n-dimensional ulim}
  The uniform limit of an $n$-uniform sequence of complexes is a separable metric space whose asymptotic and uniform dimensions are not greater than $n$. 
\end{lemma}
\begin{proof}
  Let $\mathcal{K} = \left\{ f_m \colon K_m \to K_{m+1} \right\}$ be an $n$-uniform sequence of complexes. We define the open star of a vertex $v$ in $K_m$ to be the set of all points in $K_m$ whose supports include the support of $v$ as a subset. Let $O^m_v \subset \ulim \mathcal{K}$ be the set of all threads whose $m$th coordinate is in the open star of $v$. Let $\mathcal{O}_m$ be the collection of sets $O^m_v$ over all vertices $v$ of $K_m$.

It follows from the definition that every complex is covered by open stars of its vertices. Moreover, the multiplicity of such a cover is equal to the dimension of the complex plus one. Hence $\mathcal{O}_m$ is a countable open cover of $\ulim \mathcal{K}$ and its multiplicity is not greater than $n + 1$.

Let $O^m_v \in \mathcal{O}_m$. Let $(x_k), (y_k) \in O^m_v$. By Lemma~\ref{lem:distances in a uniform complex}, the diameter of $K_k$ is at most $2^{k/2}$; hence the distance from $x_k$ to $y_k$ is at most $2^{k/2}$. By Lemma~\ref{lem:lipschitz on closed stars} and by the choice of scales of the $K_m$'s, $f_m$ is $1$-Lipschitz on the closed star of~$v$; hence the distance from $x_{m+1} = f_m(x_m)$ to $y_{m+1} = f_m(y_m)$ is not greater than the distance from $x_m$ to $y_m$, i.e. than $2^{m/2}$. By assumption, $f_m$ is barycentric, hence supports of $x_{m+1}$ and $y_{m+1}$ have non-empty intersection, so $x_{m+1}$ and $y_{m+1}$ are in a star of some vertex of $K_{m+1}$. Therefore by induction, the distance from $x_k$ to $y_k$ is not greater than $2^{m/2}$ for all $k \geq m$. Hence
\begin{equation}\tag{1}
  \mesh(\mathcal{O}_m) \leq 2^{m/2}.
\end{equation}

Let $p$ be a point in $K_m$. Let $v$ be a vertex of $K_m$ whose (only) non-vanishing coordinate is a greatest coordinate of $p$ (possibly, one of many). Since $K_m$ is an at most $n$-dimensional uniform complex of scale $2^{m/2}$, a greatest coordinate of $p$ is not less than $2^{(m-1)/2}/(n+1)$. If $q$ is a point in the complement (in $K_m$) of the open star of $v$, then the support of $q$ is disjoint from the support of $v$, so $\| p - q \|_2 \geq 2^{(m-1)/2}/(n+1)$. Hence the Lebesgue number of a cover of $K_m$ by open stars of vertices is greater than or equal to this constant. Since the metric on $\ulim \mathcal{K}$ is the $\sup$ metric on threads, we have
\begin{equation}\tag{2}
  \lambda(\mathcal{O}_m) \geq 2^{(m-1)/2}/(n+1).
\end{equation}

By the definition, $\ulim \mathcal{K}$ is a metric space. The $\mathcal{O}_m$'s are countable and by~(1) have arbitrarily small meshes, hence $\ulim \mathcal{K}$ is separable.

Let $r > 0$. Let $m$ be an integer such that $2^{m/2} < r$. The collection $\mathcal{O}_m$ covers $\ulim \mathcal{K}$ and its multiplicity is not greater than $n + 1$. By (1), the mesh of $\mathcal{O}_m$ is less than $r$. By (2), the Lebesgue number of $\mathcal{O}_m$ is positive. By  definition, $\udim \ulim \mathcal{K} \leq n$.

Let $R < \infty$. Let $m$ be an integer such that $2^{(m-1)/2}/(n+1) > R$. The collection $\mathcal{O}_m$ covers $\ulim \mathcal{K}$ and its multiplicity is not greater than $n + 1$. By (1), the mesh of $\mathcal{O}_m$ is finite. By (2), the Lebesgue number of $\mathcal{O}_m$ is greater than $R$. By definition, $\asdim \ulim \mathcal{K} \leq n$.
\end{proof}

\begin{remark}
  It is implicit in the proof of Lemma~\ref{lem:n-dimensional ulim} that the Nagata-Assouad dimension of an uniform limit of an $n$-uniform sequence is at most $n$ (in both small and large scales).
\end{remark}



\subsection{Canonical functions}
Let
\[
   \mathcal{U} = \cdots \llcurly \U{m-1} \llcurly \U{m} \llcurly \U{m+1} \llcurly \cdots
\]
be a sequence of covers of a separable metric space $X$ such that $\U{m-1}$ star-refines $\U{m}$ for each $m$. Let $N_m$ be a nerve of $\U{m}$ realized as a uniform complex of scale $2^{m/2}$. Let $p_m \colon N_m \to N_{m+1}$ be a map induced by inclusions. We say that a sequence 
\[
  \N{\U{}} =  \xymatrix{
    \cdots \ar[r]^{p_{m-2}} & N_{m-1} \ar[r]^{p_{m-1}} & N_m \ar[r]^{p_m} & N_{m+1} \ar[r]^{p_{m+1}} & \cdots
  }
\]
is a \emph{sequence of nerves associated with $\U{}$}. Clearly, $\N{\U{}}$ is $n$-uniform (cf. Lemma~\ref{lem:lipschitz on closed stars}).


Let $\pi_m \colon \ulim \N{\U{}} \to N_m$ be the projection onto the $m$th coordinate. We say that a function $\varphi$ from $X$ into $\ulim \N{\U{}}$ is a \emph{canonical function of $\U{}$} if for each $m$ the composition $\pi_m \circ \varphi$ is a canonical function from $X$ into the nerve of $\U{m}$, i.e. a function that maps $x \in X$ into a simplex of $N_m$ spanned by vertices corresponding to elements of $\U{m}$ that contain $x$. We do not require these functions to be continuous.

\begin{lemma}\label{lem:existence of a canonical function}
  If 
  \[
   \mathcal{U} = \cdots \llcurly \U{m-1} \llcurly \U{m} \llcurly \U{m+1} \llcurly \cdots
  \]
  is a sequence of covers of a space $X$ such that $\U{m}$ star-refines $\U{m+1}$ for each $m$, then $\U{}$ admits a canonical function.
\end{lemma}
\begin{proof}
  Let $\N{\U{}} = \left\{ p_m \colon N_m \to N_{m+1} \right\}$ be a sequence of nerves associated with $\U{}$. Let $x \in X$. Let $\delta^x_m$ be a simplex in $N_m$ spanned by vertices corresponding to all elements of $\U{m}$ that contain $x$. Let
\[
	D^x_m = \left\{ (x_k) \in \prod_{k \in \mathbb{Z}} N_k \colon x_m \in \delta^x_m \text{ and } x_{k+1} = p_k(x_k) \text{ for } k \geq m \right\}.
\]
Clearly, each $D_m$ is closed (in the product equipped with the $\sup$ metric) and non-empty. Let $U$ be an element of $\U{m}$ that contains $x$. By definition, $p_m$ maps the vertex $v(U)$ of $N_m$ that corresponds to $U$ onto the barycenter of a face of $N_{m+1}$ that is spanned by vertices corresponding to elements of $\U{m+1}$ that contain $U$ as a subset. But all such elements of $\U{m+1}$ contain $x$, so $p_m$ maps $v(U)$ onto a barycenter of a face of $\delta^x_{m+1}$. Since $p_m$ is affine on $\delta^x_m$, we have $p_m(\delta^x_m) \subset \delta^x_{m+1}$. Hence $D^x_m \subset D^x_{m+1}$ for each $m$. Since $p_k$ is $1$-Lipschitz on $\delta^x_k$ for each $k$, the diameter of $D^x_m$ is equal to the supremum of diameters of $N_k$'s over $k \leq m$, which by Lemma~\ref{lem:distances in a uniform complex} is equal to $2^{m/2}$. Hence by the Cantor theorem, the intersection $\bigcap_m D^x_m$ is non-empty. We define $\varphi$ by the formula $\varphi(x) \in \bigcap_m D^x_m$.
  It is a canonical function of $\U{}$, directly from the definition.
\end{proof}

\begin{remark}
  By the uniform limit theorem proved in the next section, a canonical function of a sequence of covers whose meshes converge to zero at $-\infty$ is always uniformly continuous.
\end{remark}

\subsection{Uniform limit theorem}

Theorem~\ref{thm:uniform limit theorem} is the main theorem of the third section.

\begin{theorem}\label{thm:uniform limit theorem}
  Let
  \[
   \mathcal{U} = \cdots \llcurly \U{m-1} \llcurly \U{m} \llcurly \U{m+1} \llcurly \cdots
  \]
  be a sequence of covers of a metric space $X$. Assume that for each $m$ the multiplicity of $\U{m}$ is at most $n+1$, $\U{m}$ star-refines $\U{m+1}$, the Lebesgue number of $\U{m}$ is positive and its mesh is finite.
  Let $\varphi$ be a canonical function into a uniform limit of a sequence of nerves associated with $\U{}$.

  \begin{enumerate}
  \item If $\lim_{k \to -\infty} \mesh(\U{k}) = 0$, then $\varphi$ is
    a uniform embedding with a dense image. In particular, if $X$ is
    complete, then $\varphi$ is a uniform homeomorphism.
  \item If $\lim_{k \to \infty} \lambda(\U{k}) = \infty$, then $\varphi$ is a coarse
    equivalence.
  \end{enumerate}
\end{theorem}
\begin{proof}
 Let $\N{\U{}} = \left\{ p_m \colon N_m \to N_{m+1} \right\}$ be a sequence of nerves associated with $\U{}$. Let $x_1, x_2 \in X$. If for some $k$ the distance from $x_1$ to $x_2$ is smaller than the Lebesgue number of $\U{k}$, then for each $m \geq k$, $\varphi$ composed with the projection onto the $m$th coordinate maps $x_1$ and $x_2$ into the closed star of some vertex of $N_m$. By Lemma~\ref{lem:lipschitz on closed stars}, $p_m$ is $1$-Lipschitz on closed stars of vertices. Hence the distance from $\varphi(x_1)$ to $\varphi(x_2)$ is realized either at the $k$th coordinate or earlier. By Lemma~\ref{lem:distances in a uniform complex}, the diameters of the complexes at these coordinates are bounded by $2^{k/2}$. Hence
\begin{enumerate}
    \item if $d(x_1, x_2) \leq \lambda(\U{k})$, then $d(\varphi(x_1), \varphi(x_2)) \leq 2^{k/2}$.
\end{enumerate}

If for some $k$ the distance from $x_1$ to $x_2$ is strictly greater than the mesh of $\U{k}$, then $\varphi$ composed with the projection onto the $k$th coordinate maps $x_1$ and $x_2$ into disjoint simplices in $N_k$. By Lemma~\ref{lem:distances in a uniform complex}, the distance between disjoint simplices in an at most $n$-dimensional complex of scale $2^{k/2}$ is at least $2^{k/2}/\sqrt{n+1}$. Hence
\begin{enumerate}\addtocounter{enumi}{1}
    \item if $d(x_1, x_2) > \mesh(\U{k})$, then $d(\varphi(x_1),
      \varphi(x_2)) \geq 2^{k/2}/\sqrt{n + 1}$.
\end{enumerate}

{%
\begin{figure}[ht]
\begin{center}
\resizebox{0.70\textwidth}{!}{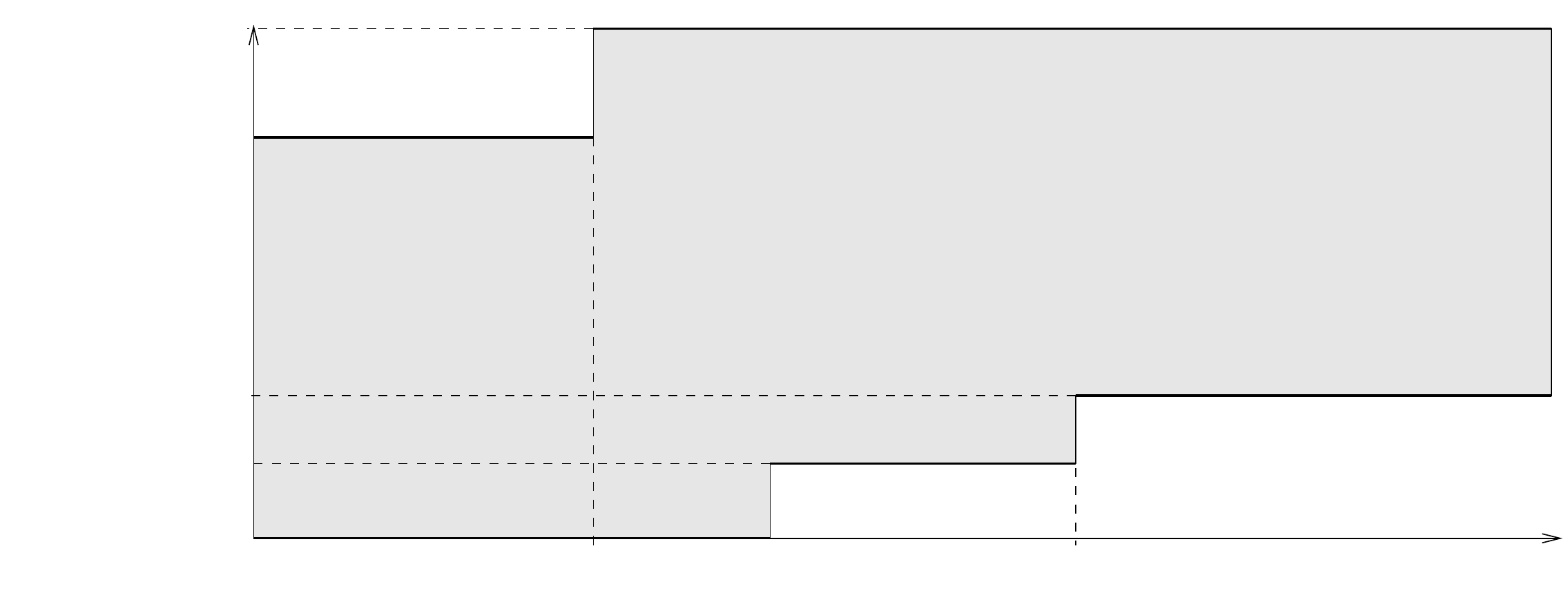}
\end{center}
\caption{{Bounds on a canonical map into $\ulim \left(N_0 \to N_1\right)$.}}
\label{canonical}
\end{figure}
}

Let
\[
	\varrho_1(r) = \sup_k \left\{ 2^{k/2}/\sqrt{n+1} \colon r > \mesh(\U{k})  \right\} 
\]

and
\[
	\varrho_2(r) = \inf_k \left\{ 2^{k/2} \colon r \leq \lambda(\U{k}) \right\}. 
\]

By (1) and (2) we have
\[
	\varrho_1(d(x_1, x_2)) \leq d(\varphi(x_1), \varphi(x_2)) \leq \varrho_2(d(x_1, x_2)).
\]

Let $\varepsilon > 0$ and let $k$ be an integer such that $2^{k/2} < \varepsilon$. If $r \leq \lambda(\U{k})$, then $\varrho_2(r) \leq 2^{k/2} < \varepsilon$. By the assumption that the Lebesgue numbers of the $\U{k}$'s are positive, we have $\lim_{r \to 0} \varrho_2(r) = 0$. Hence $\varphi$ is uniformly continuous. If $\lim_{k \to -\infty} \mesh(\U{k}) = 0$, then for each $r$ there exists $k$ such that $r \geq 2 \mesh(\U{k})$.  Hence $\varrho_1(r) > 0$ for each $r$, so $\varphi$ is an uniform embedding by the definition.

If $\lim_{k \to \infty} \lambda(\U{k}) = \infty$, then for each $r < \infty$ there exists $k$ such that $\lambda(\U{k}) > r$. Hence $\varrho_2(r) < \infty$ and $\varphi$ is large-scale uniform by the definition. Let $R < \infty$ and let $k$ be an integer such that $2^{k/2}/\sqrt{n+1} > R$. If $r \geq 2 \mesh(\U{k})$, then $\varrho_1(r) \geq 2^{k/2}/\sqrt{n+1} > R$. By the assumption that the meshes of the $\U{k}$'s are finite, we have $\lim_{r \to \infty} \varrho_1(r) = \infty$. Hence $\varphi$ is a large-scale embedding.

It follows from the definition of a canonical function that for each $k$ the projection onto the $k$th coordinate of the image of $\varphi$  intersects the closed stars of all vertices of $N_k$. By Lemma~\ref{lem:lipschitz on closed stars} the image of $\varphi$ is $2^{k/2}$ dense for each $k$. Therefore, if $\varphi$ is a uniform embedding, then it has a dense image; if $\varphi$ is a coarse embedding, then it is a coarse equivalence.
\end{proof}

\section{Construction of universal spaces}

Throughout the entire section we let $(A, a)$ denote a simplicial complex $A$ with a base vertex $a$. A map $f \colon (A, a) \to (B, b)$ is a map from $A$ into $B$ such that $f(a) = b$.


\begin{definition}
  Let $L$ be a subcomplex of a complex $K$. We let $\Int_K L$ denote the \df{combinatorial interior of $L$ in $K$}, i.e. the largest subcomplex of $L$ that is disjoint from the closure of $K \setminus L$.
\end{definition}

\begin{definition}
  Let $p \colon (\widetilde K, \tilde k) \to (K, k)$ be a barycentric map. We say that $p$ has the \df{barycentric-to-simplicial lifting property with respect to $n$-dimensional complexes} if the following condition is satisfied.

\vspace{2mm}
\begin{tabular}{ll}
  {({\sc bts}$_n$)} & 
\begin{minipage}{290pt}
  \vspace{1mm}
  \em
  For each pair $((A ,a),(B, a))$ of at most $n$-dimensional simplicial complexes and each commutative diagram
  \[
    \xymatrix@M=8pt@C=35pt{
    (\widetilde K, \tilde k) \ar[r]^{p} & (K, k)  \\
    (B, a) \ar[r]^{\subset} \ar[u]_{g} & (A, a) \ar[u]^{G} \\
    }
  \]
  such that $g$ is a simplicial embedding and $G$ is a barycentric map, there exists a simplicial embedding $\tilde g \colon (A, a) \to (\widetilde K, \tilde k)$ such that $p \tilde g = G$ and $\tilde g$ is equal to $g$ on the combinatorial interior of $B$ (taken in $A$).
  \em
  \vspace{1mm}
\end{minipage}
\end{tabular}
\vspace{2mm}  
\end{definition}

\begin{definition}
  We say that a barycentric map $p \colon (K, k) \to (K, k)$ is \df{$n$-universal} if for each 
   subcomplex $L$ of $K$ the restriction of $p$ to the preimage of $L$ has the barycentric-to-simplicial lifting property with respect to $n$-dimensional complexes.
\end{definition}

\begin{theorem}\label{thm:universal map}
  There exists an $n$-dimensional uniform simplicial complex $K$ and an $n$-universal map $p \colon (K, k) \to (K, k)$.
\end{theorem}

The proof of Theorem~\ref{thm:universal map} is postponed until the next subsection. By Theorem~\ref{thm:universal map}, there exists an $n$-universal map $p \colon (K,k) \to (K,k)$, where $K$ is an $n$-uniform simplicial complex of some fixed scale. Without any loss of generality we may assume that the scale of $K$ is equal to $1$. For each $t \in (0, \infty)$, let $i_t \colon \ell_2 \to \ell_2$ be a homothety defined by the formula $i_t(x) = tx$. Let $(K_m, k_m) = i_{2^{m/2}}((K, k))$. Let $p_m = i_{2^{(m+1)/2}} \circ p \circ i^{-1}_{2^{m/2}}$. Let
\[
 \mathcal{K}_n = \xymatrix@C=40pt{ \cdots \ar[r]^{p_{m-1}} & (K_m, k_m) \ar[r]^{p_m} & (K_{m+1}, k_{m+1}) \ar[r]^{p_{m+1}} &
    \cdots}.
\]
By the definition, $\mathcal{K}_n$ is $n$-uniform and after the natural identification of the $K_m$'s with $K$, each $p_m$ is $n$-universal.

\begin{theorem}\label{thm:limits of universal sequences}
  If $\mathbb{U}_n$ is a uniform limit of the sequence $\mathcal{K}_n$ defined above, then \[\udim \mathbb{U}_n = \asdim \mathbb{U}_n = n\] and for each separable metric space $X$ the following conditions are satisfied.
  \begin{enumerate}
  \item[a)] If $\asdim X \leq n$, then $X$ is \emph{coarsely
      equivalent} to a subset of $\mathbb{U}_n$.
  \item[b)] If $\udim X \leq n$, then $X$ is \emph{uniformly
      homeomorphic} to a subset of~$\mathbb{U}_n$.
  \end{enumerate}
\end{theorem}
\begin{proof}
  By Lemma~\ref{lem:n-dimensional ulim}, $\udim \mathbb{U}_n = \asdim \mathbb{U}_n = n$. Let $X$ be a separable metric space.
  If $\asdim X > n$ and $\udim X > n$, then we are done.
  If $\asdim X > n$ and $\udim X \leq n$, then we only have to prove b), so without loss of generality we may replace $X$ by any space uniformly homeomorphic to~$X$, in particular we may replace $(X, d)$ by $(X, \min(d, 1))$, which is of asymptotic dimension zero.
  If $\asdim X \leq n$ and $\udim X > n$, then we only have to prove a), so without loss of generality we may replace $X$ by any space coarsely equivalent with $X$, in particular we may replace $X$ by a discrete net in $X$, which is of uniform dimension zero.  
  Therefore without any loss of generality we may assume that both $\asdim X$ and $\udim X$ are less than or equal to $n$. 
  
Fix a base point $x_0$ in $X$. Since we assumed that $\asdim X$ and $\udim X$ are less than or equal to $n$, there exists a sequence $\ldots, \U{m-1}, \U{m}, \U{m+1}, \ldots$ of covers of $X$ with finite meshes, positive Lebesgue numbers, multiplicty at most $n+1$, such that $\U{m}$ star-refines $\U{m+1}$, $\lim_{m \to -\infty} \mesh(\U{m}) = 0$ and $\lim_{m \to \infty} \lambda(\U{m}) = \infty$. Moreover, we may assume (by passing to a subsequence if necessary) that $\sum_{k < m} \mesh(\U{k}) < \infty$ for each $m$.

Let $\V{m}$ be the cover of $X$ obtained from $\U{m}$ by gluing all elements of $\U{m}$ that intersect a ball $B(x_0, 2\sum_{k \leq m} \mesh(\U{k}))$ into a single element $V^0_m \in \V{m}$ (all other elements remain the same). The sequence $\ldots, \V{m-1}, \V{m}, \V{m+1}, \ldots$ satisfies the same conditions as the sequence $\ldots, \U{m-1}, \U{m}, \U{m+1}, \ldots$, i.e.
for each $m$, $\V{m}$ star-refines $\V{m+1}$, is of multiplicity at most $n+1$, has positive Lebesgue number and finite mesh, $\lim_{m \to -\infty} \mesh(\V{m}) = 0$ and $\lim_{m\to \infty} \lambda(\V{m}) = \infty$.

Let 
\[
  \N{\V{}} = \xymatrix@M=4pt@C=30pt{
    \cdots \ar[r]^{u_{m-2}} & N_{m-1} \ar[r]^{u_{m-1}} & N_m \ar[r]^{u_{m}} & N_{m+1} \ar[r]^{u_{m+1}} & \cdots
  }
\]
be the sequence of nerves associated with the sequence $\V{m}$ of covers. By Lemma~\ref{lem:existence of a canonical function}, there exists a canonical function $\kappa$ from $X$ into $\ulim \N{\V{}}$ and by Theorem~\ref{thm:uniform limit theorem}, $\kappa$ is both a uniform embedding and a coarse equivalence. Hence to finish the proof it is enough to construct an isometric embedding of $\ulim \N{\V{}}$ into $\mathbb{U}_n$. 

For each $m$, let $N^0_m = \{ v(V^0_m) \}$ be the subcomplex of the nerve $N_m$ of $\V{m}$, consisting of the single vertex corresponding to the set $V^0_m \in \V{m}$ arranged in the construction of $\V{m}$ above. Let $N^i_m = u_m^{-1}(N^0_{m+1} \cup \Int_{N_{m+1}} N^{i-1}_{m+1})$. It is a subcomplex of $N_m$ because $u_m$ is an affine map. We want to prove that for each $m$
\[\tag{*}
  N_m = \bigcup_{i \geq 0} N_m^i.
\]

To this end we shall prove by induction on $i \geq 0$ the following statement.

\begin{quote}\em
  For each $m$ the complex $N^i_m$ contains all vertices that correspond to sets $V$ in $\V{m}$ that intersect the ball $B(x_0, r^i_m)$, where
  \[ r^i_m = \sum_{k \leq m + i} \mesh(\U{k}) + \sum_{k \leq m} \mesh(\U{k}). \]
  \em
\end{quote}

For $i = 0$ the statement follows from the definition of $V^0_m$. Let $i > 0$. By the inductive assumption  $N^{i-1}_{m+1}$ contains all vertices that correspond to elements of $\V{m+1}$ that intersect $B(x_0, r^{i-1}_{m+1})$.
Let $V$ be an element of $\V{m}$ that intersects $B(x_0, r^i_m)$. 
If $V' \neq V^0_{m+1}$ is an element of $\V{m+1}$ that contains~$V$ as a subset, then~$V'$ is a subset of $B(x_0, r^i_m + \mesh(\U{m+1})) = B(x_0, r^{i-1}_{m+1})$. Hence if $V''$ intersects $V'$, then $v(V'')$ is in $N^{i-1}_{m+1}$. Hence $v(V')$ is in $\Int_{N_{m+1}} N^{i-1}_{m+1}$. By the definition, $u_m$ maps $v(V)$ to the barycenter of the simplex spanned in $N_{m+1}$ by vertices $v(V')$ such that $V \subset V'$. Hence $u_m^{-1}(N^0_{m+1} \cup \Int_{N_{m+1}} N^{i-1}_{m+1})$ contains $v(V)$.
This concludes the inductive step.

Recall that $\mathcal{K}_n = \{ p_m \colon (K_m, k_m) \to (K_{m+1}, k_{m+1}) \}$ and that each $p_m$ is $n$-universal.
Let $j^1_m \colon (N^1_m, v(V^0_m)) \to (K_{m+1}, k_{m+1})$ be a constant map onto $k_{m+1}$. By ({\sc bts$_n$}), there exists a simplicial embedding $i^1_m \colon (N^1_m, v(V^0_m)) \to (K_m, k_m)$ such that $p_m \circ i^1_m = j^1_m$. Since $u_m$ is constant on $N^1_m$, the following diagram is commutative.

\[
  \xymatrix@M=8pt{
  \cdots \ar[r]^{u_{m-2}} & \ar[d]^{i^1_{m-1}} N^1_{m-1} \ar[r]^{u_{m-1}} & N^1_m \ar[r]^{u_m} \ar[d]^{i^1_m} & \ar[d]^{i^1_{m+1}} \ar[r]^{u_{m+1}} N^1_{m+1} & \cdots \\
  \cdots \ar[r]^{p_{m-2}} & K_{m-1} \ar[r]^{p_{m-1}} & K_m \ar[r]^{p_m} & \ar[r]^{p_{m+1}} K_{m+1} & \cdots
  }
\]

For each $m$,  consider a map $j^2_m \colon (N^2_m, v(V^0_m)) \to (K_{m+1}, k_{m+1})$ defined by the formula $j^2_m = i^1_{m+1} \circ u_{m | N^2_m}$. It is a barycentric map into $K^1_{m+1}$ and $i^1_m$ is a lift of its restriction to $N^1_m$ with respect to $p_m$. 
Hence by ({\sc bts$_n$}), there exists a simplicial embedding $i^2_m \colon (N^2_m, v(V^0_m) \to (K_m, k_m)$ such that $p_m \circ i^2_m$ is equal to $j^2_m$ on $\Int_{N_m} N^2_m$. Since $u_m$ maps $N^2_{m-1}$ into the combinatorial interior of $N^2_m$, the following diagram is commutative.

\[
  \xymatrix@M=8pt{
  \cdots \ar[r]^{u_{m-2}} & \ar[d]^{i^2_{m-1}} N^2_{m-1} \ar[r]^{u_{m-1}} & N^2_m \ar[r]^{u_m} \ar[d]^{i^2_m} & \ar[d]^{i^2_{m+1}} \ar[r]^{u_{m+1}} N^2_{m+1} & \cdots \\
  \cdots \ar[r]^{p_{m-2}} & K_{m-1} \ar[r]^{p_{m-1}} & K_m \ar[r]^{p_m} & \ar[r]^{p_{m+1}} K_{m+1} & \cdots
  }
\]

We proceed recursively and construct simplicial embeddings $i^k_m \colon N^k_m \to K_m$ in the same fashion. A map $i_m \colon N_m \to K_m$ defined by a formula $i_m = \bigcup_k i^k_{m | \Int N^k_m}$ is a simplicial embedding such that $p_m i_m = i_{m+1} u_m$ for each $m$. 
Since combinatorial interiors of $N^k_m$'s cover $N_m$ for each $m$, $i_m$ is defined on entire $N_m$.
By the definition, $N_m$ and $K_m$ are uniform simplicial complexes of the same scale, hence $i_m$ is an isometric embedding.
Therefore, the $i_m$'s induce an isometric embedding of $\ulim \N{\U{}}$ into~$\mathbb{U}_n$. We are done.
\end{proof}

\subsection{Construction of a universal sequence}

Let $\mathcal{N}$ denote the set of non-empty finite subsets of $\mathbb{N}$. Let
\[
  \mathcal{D} = \left\{ D \subset \mathcal{N} \times \mathcal{N} \colon     \bigcup_{ (\sigma_1, \sigma_2) \in D} \sigma_1 \subset \bigcap_{ (\sigma_1, \sigma_2) \in D} \sigma_2 \right\}.
\]

It is an abstract simplicial complex whose vertices  $(\sigma_1, \sigma_2) \in \mathcal{N} \times \mathcal{N}$ satisfy the relation $\sigma_1 \subset \sigma_2$. Let $\Delta_1$ denote an infinite simplex of scale $1$ in $\ell_2$. Realize $\mathcal{D}$  as a subcomplex $\widetilde\Delta_1$ of $\Delta_1$. Let $r \colon \widetilde\Delta_1 \to \ell_2$ be an affine map defined on vertices of $\widetilde\Delta_1$ by the formula
\[
  r((\sigma_1, \sigma_2)) = \frac{\chi_{\sigma_2}}{\sqrt{2} | \sigma_2 |},
\]
where $\chi_{\sigma_2} \colon \mathbb{N} \to \{ 0, 1\}$ is the characteristic function of a set $\sigma_2 \subset \mathbb{N}$ and $| \sigma_2 |$ is the cardinality of~$\sigma_2$. 
Without a loss of generality we may assume that $r$ has a fixed vertex $\delta \in \Delta_1$.

\begin{lemma}\label{lem:formula for a lift}
  The map $r$ defined above is a barycentric map from $\widetilde \Delta_1$ into $\Delta_1$.  If~$f$ is a barycentric map from a simplicial complex $K$ into $\Delta_1$, then the formula
\[
    \tilde f(v) = \left( \bigcap \{ \supp f(w) \colon w  \text{ is a vertex of the closed star of } v \}, \ \supp f(v) \right)
\]
defines a lift of $f$ to a simplicial map into $\widetilde \Delta_1$, i.e. a simplicial map into $\widetilde \Delta_1$ such that $r\widetilde f = f$.
\end{lemma}
\begin{proof}
  The map $r$ defined above is an affine map into $\Delta_1$ directly from the definitions. Consider a vertex $(\sigma_1, \sigma_2) \in \mathcal{D}$. If $(\tau_1, \tau_2) \in \mathcal{D}$ is a vertex of the closed star of $(\sigma_1, \sigma_2)$, then $(\sigma_1, \sigma_2)$ and $(\tau_1, \tau_2)$ span a simplex in $\mathcal{D}$ and by the definition $\sigma_1$ is a subset of $\tau_2$. By the definition, $\tau_2 = \supp r((\tau_1, \tau_2))$. Hence the intersection of supports of images of vertices of the closed star of $(\sigma_1, \sigma_2)$ contains $\sigma_1$, so it is non-empty. Therefore $r$ maps vertices of closed stars of vertices of~$\widetilde \Delta_1$ to contiguous barycenters of~$\Delta_1$. Hence it is barycentric.

  Let $f$ be a barycentric map from a simplicial complex $K$ into $\Delta_1$. Let $\tilde f$ be an affine map into $\ell_2$ defined on a vertex $v$ of $K$ by the above formula. Since $f$ is barycentric, $\tilde f$ is a well defined map into $\widetilde \Delta_1$ and its composition with $r$ is equal to~$f$. We will show that $\tilde f$ is simplicial. Let $\{ v_1, v_2, \ldots, v_n \}$ be a set of vertices that span a simplex in $K$. Then for each $1 \leq i,j \leq n$, $v_i$ is a vertex of a closed star of $v_j$. Hence we have
\[
  \bigcup_{1 \leq i \leq n} \bigcap \{ \supp f(w) \colon w  \text{ is a vertex of the closed star of } v_i \} \subset \bigcap_{1 \leq j \leq n} \supp f(v_j),
\]
therefore the vertices $\{ \tilde f(v_1), \tilde f(v_2), \ldots, \tilde f(v_n) \}$ span a simplex in $\mathcal{D}$ directly from the definition.
\end{proof}

\begin{definition}
  Let $q \colon (\Delta_1, \delta) \to (\Delta_1, \delta)$ be a simplicial map such that $q^{-1}(v)$ is infinite for each vertex $v$ of $\Delta_1$.
  Let $(K, k)$ be a subcomplex of $\Delta_1$. Let $\delta$ be a vertex of $K$. We let
  \[
    \widetilde K = q^{-1}(r^{-1}(K))^{(n)},
  \]
  where $^{(n)}$ denotes the $n$-dimensional skeleton of a complex. We let
  \[
    p_K \colon (\widetilde K, \delta) \to (K, \delta)
  \]
   be the restriction $(r \circ q)_{|\widetilde K}$.
\end{definition}

\begin{lemma}\label{lem:universal maps}
  The map $p_K$ defined above satisfies {\sc $($bts$_n)$}.
\end{lemma}
\begin{proof}
  It is enough to check that for each $n$, the map $p_{\Delta_1} = r \circ q$ satisfies $(${\sc bts}$_n)$.
    Let $(A, B)$ be a pair of at most $n$-dimensional simplicial complexes such that $B$ is a subcomplex of $A$. 
    
The proof is divided into two parts. In the first part we prove that every simplicial map from $A$ into $\widetilde \Delta_1$ admits a lift, with respect to $q$, to a simplicial embedding; moreover, given such a lift on $\Int_A B$ we may require that the constructed lift is equal to it on $\Int_A B$. In the second part we prove that every barycentric map from $A$ into $\Delta_1$ admits a lift, with respect to $r$, to a simplicial map into $\widetilde \Delta_1$; moreover, given such a lift on $B$ we may require that the constructed lift is equal to it on $\Int_A B$. The conjunction of these two statements clearly implies the assertion of the lemma.

Let $f$ be a simplicial map from $A$ into $\widetilde \Delta_1$ and let $\tilde f$ be a lift with respect to $q$ of the restriction of $f$ to $B$. Since an inverse image under $q$ of each vertex of $\widetilde \Delta_1$ is infinite, we may may extend the lift $\tilde f$ onto all vertices of $A \setminus B$ and we may require this lift to be a simplicial embedding. Since inverse images under $q$ of all simplices of $\widetilde \Delta_1$ are isomorphic to the infinite simplex, such a lift defines a simplicial map and since it is one-to-one on vertices, it is a simplicial embedding. This concludes the first part of the proof.

Let $f$ be a barycentric map from $A$ into $\Delta_1$ and let $\tilde f$ be a lift with respect to $r$ of 
the restriction of $f$ to $B$. Let $v$ be a vertex of $\Int_A B$ and let $(\sigma_1^v, \sigma_2^v)$ in $\mathcal{D}$ denote~$\tilde f(v)$. For each vertex $v$ of $\Int_A B$, adjoin a segment to $B$ whose one endpoint is $v$ and extend $f$ onto this segment to be an affine map that maps the other endpoint to $\chi_{\sigma_1^v} / | \sigma_1^v |$. Since we adjoined segments only to vertices that lie in the combinatorial interior of $B$, the extended map is still a barycentric map into $\Delta_1$. Then, the formula given in Lemma~\ref{lem:formula for a lift} defines a lift of $f$ to a simplicial map into $\widetilde \Delta_1$ that agrees with $\tilde f$ on the combinatorial interior of $B$. We are done.
\end{proof}

\begin{proof}[Proof of Theorem~\ref{thm:universal map}]
  Let $K_0 = \{ \delta \}$ be a point. Let $p_0 \colon K_1 = \widetilde{K_0} \to K_0$ be as defined above. Define recursively
  \[
    p_i \colon K_{i+1} = \widetilde{K_i} \to K_i
  \]
  as above. Observe that for each $i$, $K_{i} \subset K_{i+1}$ and $p_{i | K_i} = p_{i-1}$, by the construction of $p_K$ and the fact that $K_0 \subset K_1$. By Lemma~\ref{lem:universal maps}, $p_1$ satisfies $(${\sc bts}$_n)$. 
 Let $K = \bigcup_{i \geq 0} K_i$ and $p = \bigcup_{i \geq 0} p_i$. We have $p(K) = K$.
 Hence $p \colon (K, \delta) \to (K, \delta)$ is $n$-universal.
\end{proof}

\begin{remark}\label{rem:asymptotic cone}
  Let $d$ be a $\sup$ metric on a uniform limit
  \[
    \ulim \xymatrix@C=35pt{
    \cdots \ar[r]^{p_{m-2}} & K_{m-1} \ar[r]^{p_{m-1}} & K_m \ar[r]^{p_m} & K_{m+1} \ar[r]^{p_{m+1}} & \cdots
  }
  \]
  of an $n$-uniform sequence $\mathcal{K}_n$ defined above. Observe that for each integer $k$, $d/{2^k}$ is a $\sup$ metric on
  \[
    \ulim \xymatrix@C=35pt{
    \cdots \ar[r]^{p_{m-k-2}} & K_{m-k-1} \ar[r]^{p_{m-k-1}} & K_{m-k} \ar[r]^{p_{m-k}} & K_{m-k+1} \ar[r]^{p_{m-k+1}} & \cdots
  }.
  \]
  Hence $(\mathbb{U}_n, d)$ is isometric to $(\mathbb{U}_n, d/2^m)$ for all integers $m$, and the asymptotic cone~\cite{gromov1993} of $\mathbb{U}_n$ is isometric to $\mathbb{U}_n$.
\end{remark}

\thanks{This paper began when the second author was a visiting professor at UNCG and concluded following the first author's visit to the Banach Center and University of Warsaw. The first author would like to express his thanks to Henryk Toru\'nczyk and the University of Warsaw for their hospitality. The second author would like to express his thanks to Alex Chigogidze and the University of North Carolina at Greensboro for supporting this research.}

\bibliography{references}

\end{document}